\documentclass[11pt,a4paper]{amsart}
\usepackage{amssymb}
\usepackage{amsmath}
\usepackage{amsfonts,a4wide}
\usepackage{bbm}
\usepackage{enumerate}
\usepackage[latin1]{inputenc}
\usepackage{shadow}
\usepackage[all]{xy}

\newtheorem{thm}[subsection]{Theorem}

\newtheorem{prop}[subsection]{Proposition}

\newtheorem{cor}[subsection]{Corollary}
\newtheorem{lemma}[subsection]{Lemma}

\newtheorem{remark}[subsection]{Remark}
\newtheorem{assumption}[subsection]{Assumption}

\newcommand{\R}{\mathbb{R}}

\newcommand{\CP}{\mathbb CP}
\newcommand{\Z}{\mathbb{Z}}
\newcommand{\C}{\mathbb{C}}

\newcommand{\N}{\mathbb{N}}
\newcommand{\U}{\mathcal{U}}

\newcommand{\W}{\mathcal{W}}

\newcommand{\cl}{\Omega^2_{\textrm{cl}}}
\newcommand{\ex}{\Omega^2_{\textrm{ex}}}
\DeclareMathOperator{\sgn}{sgn}
\DeclareMathOperator{\inter}{int}
\DeclareMathOperator{\Diff}{Diff}

\DeclareMathOperator{\colim}{colim}

\DeclareMathOperator{\Ext}{Ext}

\DeclareMathOperator{\Supp}{Supp}
\DeclareMathOperator{\Tot}{Tot}

\DeclareMathOperator{\id}{id}
\DeclareMathOperator{\Gr}{Gr}

\DeclareMathOperator{\Emb}{Emb}

\DeclareMathOperator{\Sing}{Sing}
\DeclareMathOperator{\grph}{Graph}

\begin{document}

\title{Symplectic embeddings in infinite codimension}

\author{Manuel Ara\'ujo} 
\author{Gustavo Granja}
\address{Mathematical Institute, University of Oxford, Andrew Wiles Building, Oxford OX2 6GG,UK}
\email{henriquesdea@maths.ox.ac.uk}
\address{Departamento de Matem\'atica, Instituto Superior T\'ecnico, 1049-001 Lisboa, Portugal}
\email{ggranja@math.ist.utl.pt}
\date{\today}


\begin{abstract}
Let $X$ be a union of a sequence of symplectic manifolds of increasing dimension and let $M$ be a 
manifold with a closed $2$-form $\omega$. We use Tischler's elementary method for constructing symplectic 
embeddings in complex projective space to show that the map from the space of embeddings of $M$ in $X$ to 
the cohomology class of $\omega$  given by pulling back the limiting 
 symplectic form on $X$  is a weak Serre fibration. 
 Using the same technique we prove that, if $b_2(M)<\infty$, any compact family of 
 closed $2$-forms on $M$ can be obtained by restricting a standard family of forms on 
 a product of complex projective spaces along a family of embeddings.
\end{abstract}

\maketitle

\section{Introduction}
Whitney's embedding theorem states that any manifold $M$ can be embedded in $\R^N$ when $N$ is sufficiently 
large. As $N$ increases, the connectivity of the space of embeddings 
of $M$ in $\R^N$ tends to infinity, so that the space of embeddings of $M$ in $\R^\infty = \cup_n \R^n$ (defined as the 
union of the spaces of embeddings in $\R^n$) is contractible. Thus, the embedding of  $M$ in $\R^\infty$ is
``homotopy unique".
 
It is natural to look for universal spaces into which manifolds with added geometric structure can be embedded. 
One early example of a successful such search is Narasimhan and Ramanan's construction \cite{NR} of universal 
connections. A recent example is the work of Marrero, Martínez-Torres and Padrón 
\cite{MMP} where the cases of $1$-forms and locally conformal 
symplectic structures are treated.
 
In this paper we consider the problem of finding a universal space for embedding a manifold $M$ with the
added structure given by a closed $2$-form $\omega$. There are elementary reasons why a space
$(X,\omega_X)$  playing the role of $\R^\infty$ in the first paragraph can not exist. Indeed, the
map
$$
\begin{array}{ccc}
\pi_0 \Emb(S^2,X) &  \to & \R \\
\phi & \mapsto & \int_{S^2} \phi^*\omega_X 
\end{array}
$$
would need to be surjective, precluding $X$ from being the union of a sequence 
of (second countable, Hausdorff) manifolds.

However, if we impose suitable restrictions on the cohomology class of $\omega$, universal spaces do exist.
Tischler proved in \cite{Ti} that if $\omega$ is integral (i.e. if $[\omega]$ is in the image of the map 
$H^2(M;\Z) \to H^2(M;\R)$) and $M$ is closed, then the role of $\R^N$ is played by $(\CP^N, \omega_{FS})$ where 
$\omega_{FS}$ denotes the standard Fubini-Study form on complex projective space.
In \cite{Po}, Popov showed how to extend Tischler's argument to the case when $M$ is not compact and he
also proved that any manifold with an exact $2$-form embeds in some $\R^{2N}$ with the standard 
symplectic form. Popov and Tischler's results are immediate consequences of Gromov's celebrated h-principle for
 symplectic embeddings (see \cite[3.4.2 (B), p. 335]{Gr} and \cite[Theorem 12.1.1]{EM})
  but they are of a much more elementary nature. We also note that, unfortunately, a complete proof
  (or even statement) of Gromov's parametric h-principle for symplectic embeddings does not yet seem to be available
  in the literature (cf. \cite[Remark 3.3]{GK}).

In this paper we refine Tischler and Popov's arguments so as to obtain a parametric 
and relative version of their results. It can be phrased by saying that the pullback map from the space of embeddings of 
a manifold $M$ in an infinite dimensional symplectic manifold $X$ to the space of closed $2$-forms on $M$
in a given cohomology class is a (weak Serre) fibration. We note that this is a point set statement about a map,
and as such, can not be deduced from the statement of Gromov's parametric h-principle, which is a statement
about the weak homotopy type of a space.  This result suggests studying the space of closed $2$-forms
on a manifold in terms of the topology of spaces of embeddings, and we plan to pursue this in a 
future paper.

In order to state the result precisely  we will first establish some notation that 
will be fixed for the remainder of the paper. Let $(X_n,\omega_n)$ be a sequence of 
symplectic manifolds and $i_n \colon X_n \to X_{n+1}$ be closed embeddings of positive codimension 
preserving the symplectic forms.
Given a manifold $M$ together with a closed $2$-form $\omega$, we give the space $\Emb(M,X_n)$ of 
embeddings of $M$ in $X_n$ the weak Whitney topology (see \cite[Section 2.1]{Hi}). We write
$$ \Emb_\omega(M,X_n) = \{ \phi \in \Emb(M,X_n) \colon \phi^*\omega_n = \omega \} $$
for the subspace of embeddings preserving the $2$-form, and 
$$ \Emb_{[\omega]}(M,X_n) = \{ \phi \in \Emb(M,X_n) \colon \phi^*[\omega_n] = [\omega] \} $$
for the subspace of embeddings preserving the cohomology class of the $2$-form. Note that the latter is a union
of components of the space of embeddings. We write $\Emb_\omega(M,X)$ for the union (or
colimit) of the spaces $\Emb_\omega(M,X_n)$. Recall that a subset of the union is closed if its intersection with each
of the sets $\Emb_\omega(M,X_n)$ is closed. The space $\Emb_{[\omega]}(M,X)$ is defined and
topologized analogously. We give the space $\Omega^2(M)$ of $2$-forms on $M$  the 
subspace topology determined by the weak Whitney topology on $C^\infty(M,\Lambda^2(TM))$.
The subspaces of $\Omega^2(M)$ are given the induced topology.

\begin{thm}
\label{main1}
Let $M$ be a manifold and $\omega$ a closed $2$-form on $M$. Let 
$$\pi \colon \Emb_{[\omega]}(M,X)  \to [\omega]  \subset \Omega^2(M)$$ 
be the map defined by $\pi(\phi)=\phi^*\omega_n$ (where $n$ is such that $\phi \in \Emb_{[\omega_n]}(M,X_n)$).
 Then $\pi$ is a weak Serre fibration.
\end{thm}

A map is a weak Serre fibration if it has the weak homotopy covering property with respect to finite cell complexes
(see \cite{Do} or \cite[13.1.3]{St} and \eqref{whcp} below). This means that homotopies can
 be lifted after an initial vertical homotopy, or equivalently, that initially constant homotopies can be lifted. 
Since $[\omega]$ is contractible it follows that, as long as the space  $\Emb_{[\omega]}(M,X)$ is non empty,
 any map from a finite cell complex $A$ to $[\omega]$ can be lifted along $\pi$.
  Tischler and Popov's statements are recovered by taking $A$ to be a point.

The argument used in the proof of Theorem \ref{main1} also shows that the uncountability of $H^2(M;\R)$ is,
in a sense we will now explain, the only obstruction to the existence of a universal space for
 embedding manifolds with closed two forms.

Given a product of complex projective spaces $\CP^{n_1} \times \cdots \CP^{n_k}$, we write $\omega_i$ for the 
pullback of the standard Fubini-Study form via the $i$-th projection map, and $x_i \in 
H^2(\CP^{n_1}\times \cdots \times \CP^{n_k};\Z)$ for the class obtained by pulling back the canonical generator
of $H^2(\CP^{n_i};\Z)$ via the $i$-th projection. 

If $M$ is a manifold of dimension $n$ and $n_j>\frac{n}{2}$, then homotopy classes of maps from $M$ to 
$\CP^{n_1}\times \cdots \times \CP^{n_k}$ are classified by the pullbacks of the classes $x_i$. Given 
$\alpha_1,\ldots,\alpha_b \in H^2(M;\Z)$ we write  $\Emb_\alpha(M, (\CP^{m})^b \times \CP^\infty)$ 
for the space
$$\{ \phi \in \Emb(M, (\CP^{m})^b\times \CP^\infty)
 \colon \phi^*(x_i)=\alpha_i \text{ for } 1\leq i \leq b \text{ and } \phi^*(x_{b+1}) = 0 \} $$
\begin{thm}
\label{main2}
Let $M$ be a manifold with finite second Betti number $b=b_2(M)$. Let $\Omega^2_{cl}(M)$ denote 
the space of closed $2$-forms on $M$ and let $\alpha_1,\ldots, \alpha_b
 \in H^2(M;\Z)$ be a set of classes generating $H^2(M;\R)$. Suppose $m>\frac{\dim(M)}{2}$. Then the map 
$$ \pi \colon \Emb_{\alpha}(M, (\CP^{m})^b \times \CP^\infty) \times \R^b \to \Omega^2_{cl}(M) $$
defined by
$$ \pi(\phi,\lambda) = \phi^*(\lambda_1\omega_1 + \ldots \lambda_b \omega_b + \omega_{b+1}) $$
is a \emph{surjective} weak Serre fibration.
\end{thm}

Since $\Omega^2_{cl}(M)$ is contractible, the above result implies that any map from a finite 
cell complex $A$ to $\Omega^2_{cl}(M)$ can be lifted along $\pi$. 
We interpret this as saying that the standard family of closed $2$-forms on $(\CP^\infty)^{b+1}$ 
$$  (\lambda_1, \ldots, \lambda_b) \mapsto \lambda_1 \omega_1 + \ldots + \lambda_b \omega_b + \omega_{b+1} $$
is versal for manifolds with second Betti number $\leq b$.

\subsection{Organization of the paper}
In Section 2 we prove a parametric and relative version of Tischler and Popov's result (Theorem
\ref{mainthm}) from which Theorem \ref{main1} follows immediately. We then make the 
necessary adaptations for varying cohomology class and prove Theorem \ref{main2}. 
 Section 3 gives some examples of computations of 
homotopy types of spaces of symplectic embeddings that follow easily from Theorem \ref{main1}, including 
Gal and Kedra's result (\cite [Theorem 3.2, Corollary 3.7]{GK}) that, for $\omega$ an integral closed $2$-form
on a connected manifold $M$, each connected component of  $\Emb_\omega(M,\CP^\infty)$ is homotopy equivalent to 
$(S^1)^{b_1(M)} \times \CP^\infty$. 

In order to prove Theorems \ref{main1} and \ref{main2} we need the fact that a continuous family of exact $2$-forms on 
a (not necessarily compact) manifold $M$ can be obtained by differentiating a continuous family
 of $1$-forms. Since we have not been able to find a suitable reference in the literature we have included
  a construction of a continuous deRham anti-differential in an appendix (see Theorem \ref{main}).

\subsection{Acknowledgements}
We thank João Paulo Santos for useful suggestions and David Mart\'\i nez-Torres for referring us to \cite{GK}. The 
second author gratefully acknowledges the support of the Gulbenkian Foundation.

\section{Proof of the main Theorems}

Let $(N,\omega)$ be a symplectic manifold of dimension $2n$. By a \emph{Darboux chart} for $N$ we will mean
a symplectomorphism  $\varphi \colon U \to V$ where $U \subset N$ is open and \emph{contractible}
 and $V$ is an open subset of $\R^{2n}$ with the standard symplectic form.

Recall that $(X_n,\omega_n)$ denotes a sequence of symplectic manifolds, and we assume there 
are symplectic closed embeddings $i_n\colon X_n \to X_{n+1}$ of positive
 codimension. In the examples we have in mind, the $X_n$  are either compact or $\R^{2n}$.

The idea of the proof of the Theorems of Tischler and Popov in the cases when
 $X_n=\CP^n$ (respectively $\R^{2n})$ is the following. Given a manifold $M$ and an 
integral (respectively exact) closed $2$-form $\omega$ on $M$, it is easy (by topological arguments
in the case of $\CP^n$ and trivially in the case of $\R^{2n}$) to produce a smooth 
map $f \colon M \to X_n$, for some $n$, such that $f^*[\omega_n]=[\omega]$.
One can then write $f^*\omega_n - \omega = \sum_k dh_k \wedge dt_k$ 
for some smooth functions $h_k,t_k \colon M \to \R$.
The proof proceeds by inductively adding the pairs of functions $(h_k,t_k)$ as
extra coordinates, starting with a map of the form $(f,h_1,t_1) \colon M \to X_{n+1}$
(in Darboux coordinates on $X_{n+1}$) which pulls back $\omega_{n+1}$ to $f^*\omega_n + dh_1\wedge dt_1$.
 For this to work it is important that the support of the functions $h_k,t_k$
and (in the case of $\CP^n$) their magnitude be suitably constrained.
In the end, one gets a symplectic map $g\colon M \to X_N$ for some $N>n$, 
which is necessarily an immersion. Finally, one applies Moser's lemma (requiring 
careful estimates in the non-compact case) and the density of
embeddings to turn the immersion $g$ into an embedding.

Our proof will start with an embedding instead of an arbitrary map and use the fact
that the perturbations made to the initial map will stay within embeddings. 
This will allow us to skip the application of Moser's lemma, but will otherwise consist of
obtaining parametric versions of the other steps in the proof and making use of the symplectic
neighborhood theorem to handle the case of general symplectic manifolds $X_n$.
This extra generality leads, in our view, to a considerable simplification of the proofs given in \cite{Ti,Po}.

Unless otherwise specified, the topology we give spaces of smooth functions is the weak Whitney topology.
We'll indicate the strong Whitney topology \cite[Section 2.1]{Hi} on functions by the subscript $s$. For instance
$C^\infty_s(M)$ denotes the space of smooth functions on $M$ with the strong topology. For $X$ and $Y$
topological spaces $C(X,Y)$ denotes the space of continuous maps from $X$ to $Y$ with the compact-open
topology and $C_s(X,Y)$ denotes the space of continuous maps with the strong topology (for which a basis
of open sets is given by $\{ f \in C(X,Y) \colon \grph(f) \subset U\}$ for $U$ an arbitrary open subset of $X\times Y$).

We will now state and prove our parametric and relative version of Tischler and Popov's theorem. 
Tischler and Popov's results correspond in the statement below to the cases when 
$A$ is a one point space, $U$ is the empty set and $X_n=\CP^n$ or $\R^{2n}$. 
\begin{thm}
\label{mainthm}
Let $M$ be a manifold and $\omega$ a closed $2$-form on $M$. Let 
$\pi \colon \Emb_{[\omega]}(M,X)  \to [\omega]  \subset \Omega^2(M)$ be the 
map defined by $\pi(\phi)=\phi^*\omega_n$.
Let $A$ be a compact subset of a smooth manifold and suppose given subspaces
 $B \subset U\subset A$ with $U$ open and $B$ closed. Then for each pair of
maps $f,g$ making the solid diagram commute
$$
\xymatrix{
& A  \ar[r]^-f &  \Emb_{[\omega]}(M,X) \ar[dd]^\pi \\
U \ar@{^{(}->}[ur] \ar@{^{(}->}[dr] & & \\
& A  \ar[r]^-g \ar@{-->}[ruu]^-h &  [\omega] 
}
$$
there exists a lift $h$ of $g$ which is homotopic to $f$ relative to $B$.
\end{thm}

The condition imposed on the space $A$ in Theorem \ref{mainthm} is an artifact of the proof of Lemma \ref{functions}
below. The main thing to keep in mind is that the class of allowable spaces $A$ includes all finite cell complexes 
(and even compact ENRs).

\begin{lemma}
\label{functions}
Let $A$ be a compact subset of a smooth manifold and $B\subset A$ a closed subset.
Let $M$ be a manifold of dimension $n$, and $\eta\colon A\to \Omega^1(M)$ be
 a continuous map. Assume $\{W_\alpha\}_{\alpha \in I}$ is a locally finite open cover of $A \times M$
such that, for each $\alpha$, there is a coordinate neighborhood $U$ of $M$ so that $W_\alpha \subset A \times U$.
Then there exist continuous functions $h^r_\alpha,t^r_\alpha \colon A \to C^\infty(M)$ for
 $r=1,\ldots,n$ and $\alpha \in I$, such that 
\begin{enumerate}[(i)]
\item For each $r=1,\ldots, n$ and $\alpha \in I$, the functions $(z,x) \mapsto h^r_\alpha(z)(x)$ 
and $(z,x) \mapsto t^r_\alpha(z)(x)$ 
are supported on $W_\alpha$,
\item If $\eta(z)=0$ for all $z$ in a neighborhood of $B$, then $h_\alpha^r(z)=t^r_\alpha(z)=0$ for all $z\in B$,
\item $d\eta(z) = \sum_{\alpha \in I} \sum_{r=1}^n dh^r_\alpha(z)\wedge dt^r_\alpha(z)$.
\end{enumerate}
\end{lemma}
\begin{proof}
Let $\{\rho_\alpha\}$ be a smooth partition of unity subordinate to the cover $\{W_\alpha\}$ of 
$A\times M$ (this makes sense because $A$ can be regarded as a closed subset of a smooth manifold). 
We can write
$\rho_\alpha \eta = \sum_{r=1}^n h^r_\alpha(z) ds^r_\alpha$ where $s^r_\alpha$ are local coordinates 
on some coordinate neighborhood $U$ so that $W_\alpha \subset A \times U$ and 
$h^r_\alpha \colon A\to C^\infty(M)$ are continuous functions such that the support of $(z,x) \mapsto 
h^r_\alpha(z)(x)$ is contained in the support of $\rho_\alpha$.

Let $\phi_\alpha \colon A \times M \to \R$ be a smooth cut-off function supported on $W_\alpha$ which is equal to $1$
on an open set containing the support of $\rho_\alpha$. 
Suppose $U$ is a neighborhood of $B$ where $\eta$ vanishes and let $\psi \colon A \to [0,1]$ be a continuous
function which is supported on $U$ and equal to $1$ on $B$. 
Setting $t^r_\alpha(z) = (1-\psi(z)) \phi_\alpha(z, \cdot) s^r_\alpha$  we obtain
the required expression: $d\eta(z) = \sum_\alpha d( \rho_\alpha\eta(z) )  = \sum_\alpha \sum_{r=1}^n 
dh^r_\alpha(z) \wedge dt^r_\alpha(z)$.
\end{proof}

We will need the fact that the functions produced in the previous lemma can be made arbitrarily small.
The next lemma will ensure this can be arranged. We note that the proof of this point given 
in \cite[Lemma 2 (3),(4)]{Ti} is mistaken.

\begin{lemma}(cf. \cite[12.1.5]{EM})
\label{twisting}
Let $D^2_s$ denote the disk of radius $s$ in $\R^2$.
Given $r,R>0$, there exists a smooth symplectic map $D^2_R \to D^2_r$ sending the origin to itself.
\end{lemma}
\begin{proof}
A map with the required properties can be obtained by composing a map $\C\setminus \{0\} \to \C \setminus\{0\}$
 of the form $z \mapsto \frac{z^{N+1}} {N|z|^N}$ (for a suitable $N \in \N$) with translations of $\C$ on both sides.
\end{proof}

\begin{proof}[Proof of Theorem \ref{mainthm}]
We will use capital letters to denote adjoint maps. For instance $ F \colon A \times M \to X $
denotes the map $(z,x) \mapsto f(z)(x)$. Since $A$ is compact and $\Emb_{[\omega]}(M,X)$ has the colimit topology,
 there exists $N$ such that the image of $F$ is contained in $X_N$. 

Let  $n=\dim M$, let $k$ be the dimension of a smooth manifold containing $A$, let
$d_j = \dim X_j$, and set $C=n(n+k+1)$.
 Let $\U_N = \{U_{N,\beta}\}_{\beta \in J}$ be a locally finite cover of $X_N$ 
by Darboux coordinate neighborhoods and $\phi_{N,\beta} \colon U_{N,\beta} \to \R^{2d_N}$
be corresponding Darboux charts.

By the symplectic neighborhood theorem we can pick successive extensions of 
the Darboux coordinates on $X_N$ to the manifolds $X_j$ which we denote
$$ \phi_{j,\beta} \colon U_{j,\beta} \to \R^{2d_j},  \quad \quad j=N+1, \ldots, N+C$$
The families $\U_j =  \{ U_{j,\beta}\}_{\beta \in J}$ are open covers of $X_N$ in $X_j$ 
satisfying $\U_j \cap X_{j'} = \U_{j'}$ for $j'<j$.
We will use the charts $\phi_{j,\beta}$ to construct a map $H \colon A \times M \to X_{N+C}$ whose adjoint will lift $g$.

Let $\W=\{W_\alpha\}_{\alpha \in I}$ be a locally finite refinement of the cover $F^{-1}(\U_N)$ 
with the property that each element of $\W$ has compact closure and is
contained in $A \times U$ for some coordinate chart $U$ on $M$.
We'll denote the refinement function by $\psi \colon I \to J$.
 Since $A\times M$ has covering dimension less than or equal to 
 $n+k$,  according to \cite[Lemma 2.7]{Mu} we can (by further
refining $\W$ if necessary)  partition the indexing set $I$ of $\W$ as
$$I = I_0 \coprod \ldots \coprod I_{n+k}$$
so that  if $\alpha, \alpha'\in I_i$ then $W_\alpha\cap W_{\alpha'} = \emptyset$.

By Theorem \ref{main} there is a continuous map $\eta \colon A \to \Omega^1(M)$ which vanishes on $U$
and satisfies, 
$$ d(\eta(z)) = g(z)-\pi(f(z)) = g(z)-f(z)^*\omega_N $$
for all $z\in A$. Lemma  \ref{functions} then provides continuous functions 
$$ h^r_\alpha, t^r_\alpha \colon A \to C^\infty(M), \quad  r=1,\ldots,n; \ \alpha \in I$$
such that 
\begin{enumerate}
\item $S_\alpha= \cup_{r=1}^n \left( \Supp H^r_{\alpha} \cup \Supp T^r_{\alpha}\right)$ is contained in $W_{\alpha}$,
\item $h^r_{\alpha}(z)=t^r_\alpha(z)=0$ for $z \in B$,
\item $d\eta(z) = \sum_{r,\alpha} dh^r_{\alpha}(z) \wedge dt^r_{\alpha}(z).$
\end{enumerate}
We are now in a position to complete the proof by constructing a homotopy 
$$ k \colon A \times [0,C] \to \Emb_{[\omega]}(M,X) $$
such that
\begin{enumerate}
\item $k(z,0)=f(z)$ for all $z \in A$,
\item $k(z,t)=f(z)$ for all $z \in B$,
\item  $\pi k(z,C)=g(z)$, for all $z \in A$.
\end{enumerate}
For each $m \in \{0,\ldots , n+k\}$ and $r \in \{1,\ldots n\}$ define
 $h^r_m =\sum_{\alpha \in I_m} h^r_\alpha$ and  $t^r_m =\sum_{\alpha \in I_m} t^r_\alpha$. Then
$$ d\eta(z) = \sum_{r=1}^n \sum_{m=0}^{n+k} dh^r_{m}(z) \wedge dt^r_{m}(z). $$

Note that the maps
$$(H^r_m,T^r_m) \colon A \times M \to \R^2 $$
are supported on a disjoint union of compact subsets of $A\times M$ indexed by the set $I_m$. Thus, 
if $\mathcal N$ is any neighborhood of $0$ in $C_s(A\times M,\R^2)$,
 Lemma \ref{twisting} allows us to construct from $h^r_m, t^r_m$ maps 
$$(\overline{H}^r_m,\overline{T}^r_m) \colon A \times  M \to \R^2 $$
so that 
\begin{enumerate}
\item $\overline{H}^r_m,\overline{T}^r_m \in \mathcal N$,
\item $\Supp (\overline{H}^r_m,\overline{T}^r_m) \subset \Supp (H^r_m,T^r_m) (\subset \cup_{\alpha \in I_m} S_\alpha)$,
\item $(\overline{h}^r_m,\overline{t}^r_m) \in C(A,C^\infty(M,\R^2))$,
\item $dh^r_m(z) \wedge dt^r_m(z) = d\overline{h}^r_m(z) \wedge d\overline{t}^r_m(z)$ for all $z \in A$.
\end{enumerate}

The homotopy $k \colon A \times [0,C] \to C^\infty(M,X)$ is constructed inductively as follows. 
We pick the neighborhood $\mathcal N$ of $0$ in $C_s(A \times M,\R^2)$ so small that the map 
$$ k_{|A \times [0,1]}  \colon A \times [0,1] \to C^\infty(M,X_{N+1}) $$
given by the formula
$$ k(z,s)(x) = 
\begin{cases}
\phi^{-1}_{N+1,\psi(\alpha)}(\phi_{N,\psi(\alpha)}(F(z,x)), s\overline{H}^1_0(z,x), s\overline{T}^1_0(z,x), 0, \ldots, 0) 
&  \text{if } (z,x) \in \cup_{\alpha \in I_0} S_\alpha   \\
F(z,x) & \text{otherwise.}
\end{cases}
$$
is well defined and satisfies
\begin{enumerate}
\item $k(z,s) \in \Emb_{[\omega]}(M,X_{N+1})$ for all $z \in A$ and $s \in [0,1]$,
\item $k(z,s)(S_\gamma) \subset U_{N+1,\psi(\gamma)}$ for all $z \in A, s \in [0,1]$ and $\gamma \in I$.
\end{enumerate}
This is possible, because (1) and (2) can be enforced by imposing a finite number of bounds on the 
functions $\overline{H}^1_\alpha,\overline{T}^1_\alpha$ on the compact set $S_\alpha$ for each $\alpha \in I_0$.

Then setting $g_1(z)=k(z,1)$ we have
$$ g_1(z)^*\omega_{N+1} = f(z)^*\omega_N + dh^1_0(z) \wedge dt^1_0(z). $$
Since $G_1(S_\gamma) \subset U_{N+1,\psi(\gamma)}$, we can proceed with the construction 
of the homotopy in the same way.
Namely we obtain $g_2=k(z,2) \colon A \to \Emb_{[\omega]}(M,X_{N+2})$ with $g_2(z)^*\omega_{N+2} = 
g_1(z)^*\omega_{N+1} + dh^2_0(z) \wedge dt^2_0(z)$ and $G_2(S_\gamma)
\subset U_{N+2,\psi(\gamma)}$ for all $\gamma \in I$, etc.
Setting $h(z)=k(z,C)$ we'll have $h(z)^*\omega_{N+C} = f(z)^*\omega_N + d\eta(z) = g(z)$ as required.
\end{proof}
 
\begin{remark}
The construction given in the previous proof shows that the lift $h$ can be chosen so 
that $H$ approximates $F$ arbitrarily in $C_s(A\times M,X_{N+C})$.
\end{remark}

We can now prove Theorem \ref{main1}. First recall from \cite{Do} that a
 map $\pi \colon E \to X$ is said to have the \emph{weak homotopy lifting property} with respect to a space $B$ 
if given a commutative diagram 
\begin{equation}
\label{whcp}
\xymatrix{
B \times 0 \ar@{^{(}->}[d] \ar[r]^-k & E \ar[d]^\pi \\
B \times [0,1] \ar[r]_-H & X
}
\end{equation}
there exists a homotopy $\tilde{H} \colon B \times [0,1] \to E$ such
 that $\pi \tilde{H} = H$ and $\tilde{H}_0$ is homotopic
to $k$ as maps over $\pi$ (i.e. $\tilde{H}_0$ is vertically homotopic to $k$). Clearly this
 is equivalent to the usual homotopy lifting property for homotopies $H$
which are initially constant (i.e. which, for some $\epsilon>0$, satisfy $H(b,t)=H(b,0)$ for all $b\in B$ and 
$t\leq \epsilon$).
See \cite{Do} and  \cite[13.1.3]{St} for more on the weak homotopy lifting property. A map $\pi$ which has
the weak homotopy lifting property with respect to all finite cell complexes is called a \emph{weak Serre fibration}.
Since any finite cell complex can be embedded in Euclidean space, the following Corollary finishes the 
proof of Theorem \ref{main1}.

\begin{cor}
\label{mainthm2}
The map $\pi \colon \Emb_{[\omega]}(M,X) \to [\omega] \subset \Omega^2(M) $ defined by 
$\pi(\phi)=\phi^*\omega_n$ has the weak homotopy lifting property with respect to compact subsets of smooth 
manifolds.
\end{cor}
\begin{proof}
Let $B$ be a compact subset of a smooth manifold. Given a homotopy 
$H\colon B \times [0,1] \to [\omega]$ which is constant
in the interval $[0,\epsilon]$ and a lift $k \colon B \to \Emb_{[\omega]}(M,X)$ of $H_0$, apply Theorem \ref{mainthm}
with $A=B \times [0,1], U=B \times [0,\epsilon[,B=B\times 0, g=H$ and $f(a,t)=k(a)$.
\end{proof}

We end this section by recording the following immediate consequence of Theorem \ref{main1}
\begin{cor}
\label{fiber}
Let $M$ be a manifold and $\omega \in \Omega^2(M)$ be a closed $2$-form. Then the inclusion 
$$ \Emb_{\omega}(M,X) \to \Emb_{[\omega]}(M,X) $$ is a weak homotopy equivalence. 
\end{cor}
\begin{proof}
The usual construction of the long exact sequence of homotopy groups in a fibration goes through for maps 
satisfying the weak homotopy lifting property with respect to closed balls. Since $[\omega]$ is convex, the 
result follows from Theorem \ref{main1}.
\end{proof}

\subsection{Proof of Theorem \ref{main2}}
We'll make use of the setup described just  before the statement 
of Theorem \ref{main2}. We write 
$$\psi \colon \cl(M) \to \R^b \cong H^2(M;\R) $$
for the map defined by the equation
$$ [\omega] = \psi_1(\omega) \alpha_1 + \ldots + \psi_b(\omega) \alpha_b $$
where $\alpha_i \in H^2(M;\Z)$ are classes whose images span $H^2(M;\R)$ (and we are still writing
$\alpha_i$ for their images in $H^2(M;\R)$).  We note that $\psi$ is continuous 
as its coordinate functions are given by integrating over appropriate $2$-cycles.
 Theorem \ref{main2} is an immediate consequence of the following version of Theorem \ref{mainthm}. 

\begin{prop}
\label{maintheorem2}
Let $M$ be a manifold with finite second Betti number $b=b_2(M)$. Let $\Omega^2_{cl}(M)$ denote 
the space of closed $2$-forms on $M$ and let $\alpha_1,\ldots, \alpha_b
 \in H^2(M;\Z)$ be a set of classes generating $H^2(M;\R)$. Suppose $m>\frac{\dim(M)}{2}$,
 and consider the map 
$$ \pi \colon \Emb_{\alpha}(M, (\CP^{m})^b \times \CP^\infty) \times \R^b \to \Omega^2_{cl}(M) $$
defined by
$$ \pi(\phi,\lambda) = \phi^*(\lambda_1\omega_1 + \ldots \lambda_b \omega_b + \omega_{b+1}) $$
Let $A$ be a compact subset of a smooth manifold and suppose given subspaces
 $B \subset U\subset A$ with $U$ open and $B$ closed. Then for each pair of
maps $f,g$ such that
$$ \pi_2 f = \psi g$$
and such that the solid diagram commutes
$$
\xymatrix{
& A  \ar[r]^-f &  \Emb_{\alpha}(M,(\CP^m)^b \times \CP^\infty)  \times \R^b \ar[dd]^\pi \ar[ddr]^{\pi_2} &   \\
U \ar@{^{(}->}[ur] \ar@{^{(}->}[dr] & & & \\
& A  \ar[r]^-g \ar@{-->}[ruu]^-h &  \Omega^2_{cl}(M) \ar[r]_{\psi} & \R^b \\
}
$$
there exists a lift $h$ of $g$ which is homotopic to $f$ relative to $B$.
\end{prop}
\begin{proof}
Write $f=(f_1,f_2)$ with  $f_1 \colon A \to  \Emb_{\alpha}(M,(\CP^m)^b \times \CP^\infty)$
and $f_2 =\psi g \colon A \to \R^b$. Similarly, write $h=(h_1,h_2)$ 
and $k=(k_1,k_2)$ for the homotopy  between $f$ and $h$ whose existence the Proposition asserts.

We set $h_2(a) = k_2(a,t)= \psi g(a)$ for all $a,t$. 
As in the proof of Theorem \ref{mainthm}, the map $f_1$ factors through
$\Emb_{\alpha}(M,(\CP^m)^b \times \CP^N)$ for some $N$. 
It suffices to construct a deformation $k_1$ of $f_1 \colon A \to \Emb_{\alpha}(M,(\CP^m)^b \times \CP^N)$ 
so that $\pi (k_1( a, 1), \psi(g(a))) = g(a)$.  Letting $X_n=(\CP^m)^b \times \CP^n$, 
the construction of the homotopy in the proof of Theorem \ref{mainthm} applies verbatim.
\end{proof}

\begin{proof}[Proof of Theorem \ref{main2}]
The proof that $\pi$ is a weak Serre fibration is exactly the same as in Corollary \ref{mainthm2} (define the second 
coordinate of $f(a,t)$ to be $\psi H(a,t)$). Since $\cl(M)$ is path connected and 
$\Emb_{\alpha}(M,(\CP^m)^b \times \CP^\infty) \neq \emptyset$, it follows that $\pi$ is surjective.
\end{proof}

\begin{remark}
It would be more satisfying if the second Betti number of the ambient 
space appearing in Theorem \ref{main2}  were equal to $b$. This can be arranged at the expense of excluding the
$0$ cohomology class. Writing $\ex(M)$ for the space of exact forms, Proposition \ref{maintheorem2} and 
Theorem \ref{main2} continue to hold if we replace $\Emb_\alpha(M, (\CP^m)^b\times \CP^\infty)$ with 
$ \Emb_\alpha(M, (\CP^\infty)^b)$ and $\cl(M)$ with $\cl(M)\setminus \ex(M)$. It suffices to make the following
changes in the proof: set $X_n = (\CP^n)^b$ and pick Darboux charts given by products of Darboux 
charts on $\CP^n$. In the first step of the inductive construction of the homotopy 
$k \colon A \times [0,C] \to \Emb_{\alpha}(M,X)$  
replace the expression
 $(F(z,x), s\overline{H}_0^1(z,x),s\overline{T}_0^1(z,x), 0, \ldots, 0)$
by 
$$(F(z,x), s \mu_1(z)\overline{H}_0^1(z,x),s \mu_1(z) \overline{T}_0^1(z,x),  s \mu_2(z) \overline{H}_0^1(z,x), 
\ldots,  s \mu_b(z) \overline{T}_0^1(z,x))$$
where 
$$ \mu_i(z)= \frac{\psi_i(g(z))}{\sum_{i=1}^b \psi_i(g(z))^2} $$
Proceed with the construction of the homotopy in the same way.
\end{remark}

\section{Examples of spaces of symplectic embeddings}
\label{apps}

In order to compare the homotopy type of the space of embeddings with the homotopy type of the space of 
continuous functions it is convenient to make the following additional assumption on the symplectic embeddings
$i_n \colon X_n \to X_{n+1}$.
\begin{assumption}
\label{connectivity}
The connectivity of
the symplectic embeddings $i_n \colon X_n \to X_{n+1}$ goes to $\infty$ with $n$ (i.e. for any given $N$,
 there exists $m$ so that $i_{n*} \colon \pi_k(X_n,*) \to \pi_k(X_{n+1},*)$ is an isomorphism
for all  basepoints, all $k\leq N$ and $n\geq m$).
\end{assumption}
Recall that for $X,Y$ topological spaces, $C(X,Y)$ denotes the space of continuous maps with the compact 
open topology. Under Assumption \ref{connectivity}, it follows from Whitehead's Theorem (see for instance
\cite[Theorem 5.1.32]{AGP}) that the canonical map
\begin{equation} 
\label{colimcont}
\colim_n C(M,X_n) \to C(M,X)
\end{equation}
is a weak equivalence for any manifold $M$. Moreover the cohomology classes $[\omega_n] \in H^2(X_n;\R)$ 
determine a unique element $[\omega_X] \in H^2(X;\R)$ pulling back to $[\omega_n]$ under the inclusions.
\begin{prop} 
\label{contmaps}
Let $M$ be a manifold, $\omega$ be a closed $2$-form on $M$ and assume the symplectic 
embeddings $i_n \colon X_n \to X_{n+1}$ satisfy Assumption \ref{connectivity}. Let 
$$C_{[\omega]}(M,X) = \{ f \in C(M,X) \colon f^*[\omega_X]=[\omega] \}.$$
Then the inclusion 
$$ \Emb_{\omega}(M,X) \to C_{[\omega]}(M,X)$$ 
is a weak homotopy equivalence. 
\end{prop}
\begin{proof}
First recall that a map $\phi \colon X \to Y$ is a weak homotopy equivalence, if and only if for all $k\geq 0$ and 
all commutative squares
\begin{equation}
\label{lifting}
\xymatrix{
S^{k-1} \ar[r]^-{f} \ar@{^{(}->}[d] & X  \ar[d]^\phi \\
D^k \ar[r]_-{g} \ar@{-->}[ur]^h & Y 
}
\end{equation}
(where $D^k$ denotes the closed unit ball in $\R^k$) there exists a lift $h$ making the upper 
triangle commute and the lower triangle commute up to homotopy relative to
$S^{k-1}$ (see for instance \cite[Lemma in section 9.6]{Ma}).

In view of Corollary \ref{fiber} it suffices to show that each of the inclusions
 $i_1 \colon \Emb_{[\omega]}(M,X) \hookrightarrow C^\infty_{[\omega]}(M,X)$ and 
$i_2 \colon C^\infty_{[\omega]}(M,X) \hookrightarrow C_{[\omega]}(M,X)$ are weak equivalences
 (where spaces of smooth maps to $X$ mean the union of the corresponding spaces of
  smooth maps to the manifolds $X_n$).

In the case of $i_2$, since the map \eqref{colimcont} is a weak equivalence, it is enough to show 
that the inclusions $i_2 \colon C^\infty(M,X_n) \hookrightarrow C(M,X_n)$
are weak equivalences for all $n$. But the required lifts in \eqref{lifting} 
exist by the density of smooth maps in continuous
 maps in the strong topology (cf. \cite[Theorem 2.6]{Hi}).

In order to prove that the map $i_1$ satisfies the lifting conditions \eqref{lifting},
 note that compactness of $D^k$ implies that the image of $g$ is contained in $C^\infty(M,X_n)$
 for some $n$.  As long as the dimension of $X_n$ is large enough, 
the density of embeddings in smooth maps in the strong topology
(cf. \cite[Theorem 2.13]{Hi}) allows us to construct the required lift.
The reader is referred to \cite[Sections 3.2 and 3.3]{Ar} for more details of these arguments.
\end{proof}

\subsection{Examples}

\subsubsection{Embeddings in $\R^\infty$.} Let $X_n=\R^{2n}$, let $\omega_n$ denote the standard symplectic
 form on $X_n$ and $i_n \colon \R^{2n} \to \R^{2n+2}$ denote the canonical inclusions. Assumption 
\ref{connectivity} is obviously satisfied. In this case, Proposition \ref{contmaps} says that for $(M,\omega)$
 a  manifold with an exact $2$-form, the space $\Emb_\omega(M,\R^\infty)$ is contractible. 
For instance if $\omega=0$, this says that the space of isotropic embeddings of $M$ in $\R^\infty$
 is contractible.
Since the group $\Diff_\omega(M)$ of diffeomorphisms of $M$ which preserve $\omega$ acts freely 
on $\Emb_\omega(M,\R^{\infty})$, we see that, when $M$ is compact, the space of
 submanifolds of $\R^\infty$ which are
 diffeomorphic to $(M,\omega)$ provides a model for the classifying space $B\Diff_\omega(M)$.

\subsubsection{Embeddings in $\CP^\infty$.}  \label{cpinfty}
Let $X_n=\CP^n$, with $\omega_n$ the standard Fubini-Study form, and $i_n 
\colon \CP^n \to \CP^{n+1}$ denote the canonical inclusions.  Assumption \ref{connectivity} is satisfied, so
if $(M,\omega)$ is a manifold with a closed $2$-form, Proposition \ref{contmaps} says that the 
inclusion
$$\Emb_{\omega}(M,\CP^\infty) \hookrightarrow C_{[\omega]}(M,\CP^\infty)$$
is a weak homotopy equivalence (cf. also \cite[Theorem 3.2]{GK} where this result is proved for $M$ compact).

Since $\CP^\infty$ is an Eilenberg-Maclane space $K(\Z,2)$, the space $C(M,\CP^\infty)$ is
weakly equivalent to a product of Eilenberg-MacLane spaces
$$ C(M,\CP^\infty) = H^2(M;\Z) \times K(H^1(M;\Z),1) \times K(H^0(M;\Z),2).$$
This is a standard computation but since we were unable to find a suitable reference we'll sketch an argument:
$\CP^\infty$ is weakly equivalent to a topological abelian group $X$ (see for instance \cite[Corollary 6.4.23]{AGP}), hence $C(M,\CP^\infty)$ is weakly equivalent to $C(M,X)$. The singular complex $\Sing(C(M,X))$ is a
simplicial abelian group, and so, by the Dold-Kan correspondence and elementary properties of chain
complexes of abelian groups, is weakly equivalent to a simplicial abelian group $G=\prod_{k \geq 0} G_k$
with each $G_k$ a simplicial abelian group with only one simplex in degrees less than $k$ and 
$\pi_i(G_k)=0$ for $i\neq k$.  The geometric realization $|G|$ of $G$ is weakly equivalent to $C(M,\CP^\infty)$
and also to $\prod_k |G_k| = \prod_k K(\pi_k(G_k),k)$. The computation of the 
homotopy groups of $C(M,\CP^\infty)$ in terms of the cohomology of $M$ follows from the fact that $\CP^\infty$ 
classifies degree $2$ cohomology (cf. \cite[Lemma 3.6]{GK}).

If $\omega$ is an integral form,
 then the space $C_{[\omega]}(M,\CP^\infty)$ is the union
of the components corresponding to the class $[\omega] \in H^2(M;\R)$, so
$$ \Emb_\omega(M,\CP^\infty) \simeq \Ext(H_1(M),\Z) \times (S^1)^{\beta_1(M)} \times (\CP^\infty)^{\beta_0(M)}$$
where $\beta_0(M)$ denotes the number of connected components of $M$ and $\beta_1(M)$ the first Betti number
(which may be infinite).


In particular, if $(M,\omega)$ is simply connected (and integral), $\Emb_\omega(M,\CP^\infty) \cong \CP^\infty$
with the weak equivalence induced by evaluation at a given point in $M$. Once we fix a base point in 
$\CP^\infty$ the space $\Emb_{\omega,*}(M,\CP^\infty)$ of pointed embeddings is contractible and so, 
for $M$ compact, the quotient by the space $\Diff_{\omega}(M,*)$ of diffeomorphisms
 preserving $\omega$ and fixing the basepoint
provides a model for the classifying space $B\Diff_\omega(M,*)$.

\subsubsection{Embeddings in products of copies of $\CP^\infty$}
Let $M$ be a manifold with finite second Betti number and $\omega$ a closed $2$-form on $M$.  Let $r=r([\omega])$
be the least positive integer such that $[\omega]$ can be written as a real linear combination of $r$ elements in
$H^2(M;\Z)$. We can pick  $\alpha_i \in H^2(M;\Z)$ 
and $\lambda_i \in \R\setminus\{0\}$ for $i=1, \ldots, r$ such that 
$$ [\omega] = \lambda_1 \alpha_1 + \ldots + \lambda_r \alpha_r.$$ 
Consider the sequence of symplectic manifolds $X_n=(\CP^n)^r$ with symplectic form 
$\omega_n = \sum_{i=1}^r \lambda_i \omega_i$ where $\omega_i$ denotes the pullback of the standard 
Fubini-Study form via the $i$-th projection. Assumption \ref{connectivity} is satisfied
and therefore, by Proposition \ref{contmaps}, the classes $\alpha_i$ determine a connected
 component of  $\Emb_{\omega}(M,(\CP^\infty)^r)$. As in the previous example, the 
 weak homotopy type of this connected component is $(S^1)^{\beta_1(M)r} \times 
(\CP^\infty)^{\beta_0(M)r}$.

\subsubsection{Embeddings in $BU$}
Let $X_n=\Gr_n(\C^{2n})$ be the Grassmann manifold of $n$-dimensional complex subspaces of $\C^{2n}$. 
The manifolds $X_n$ admit canonical Pl\"ucker embeddings in $\CP^{2n \choose n}$ which are compatible with 
the standard inclusions $i_n \colon X_n \to X_{n+1}$ and
 $\CP^{2n \choose n} \subset \CP^{2(n+1) \choose n+1}$ 
(for a suitable ordering of the coordinates in projective space).
 The restriction of the Fubini-Study forms on the projective spaces
give rise to canonical symplectic forms $\omega_n$ on $X_n$ such that $i_n^* \omega_{n+1}=\omega_n$.
The form $\omega_n$ is integral and its cohomology class generates $H^2(X_n;\Z)\cong \Z$. 

The colimit $X$ of the inclusions $i_n$ is the classifying space for stable complex bundles usually denoted $BU$.
Assumption \ref{connectivity} is satisfied, so given a manifold $M$ and a closed 
$2$-form $\omega$ on $M$, Proposition \ref{contmaps} identifies the weak homotopy type of 
$\Emb_\omega(M,BU)$ with that of $C_{[\omega]}(M,BU)$.
The homotopy groups of this space can be computed in terms of (complex, representable) $K$-theory of $M$ since
$BU$ classifies reduced complex $K$-theory.

In more detail, there is a (total Chern class) map $c \colon BU \to K(\Z,2) \times K(\Z,4) \times \cdots$ classifying the
stable Chern classes which is an equivalence on rational cohomology. The composite of $c$ with the
projection onto $K(\Z,2) \cong \CP^\infty$
can be identified with the inclusion of $BU$ in $\CP^\infty$ determined by the Pl\"ucker embeddings. 
Therefore the set of connected components of $C_{[\omega]}(M,BU)$ is
$$\pi_0(C_{[\omega]}(M,BU))= \{ \alpha \in \tilde{K}^0(M) \colon c_1(\alpha)=[\omega] \}. $$
For this set to be non-empty, $[\omega]$ must of course be an integral class and since the map $BU \to K(\Z,2)$
admits a section, this is in fact the only condition needed to ensure the set is non-empty.

Since $BU$ has a group-like multiplication, all the connected components of $C_{[\omega]}(M,BU)$ are weakly
equivalent and the homotopy groups agree with those of the connected component corresponding to the 
constant map from $M$ to $BU$. Hence for each connected component $\alpha$ we have 
$$ \pi_i(\Emb_{\omega}(M,BU)_\alpha) = \begin{cases} K^1(M) & \text{if } i \text{ is odd,} \\
\tilde{K}^0(M) & \text{if } i>0 \text{ is even.}
\end{cases}$$

\appendix
\section{A continuous deRham anti-differential}

Let $M$ be a (second countable, Hausdorff) smooth manifold without boundary.
We give spaces of differential forms on $M$ the weak Whitney topology.
The aim of this appendix is to prove the following result which was used in the proof 
of Theorem \ref{mainthm}. 
\begin{thm}
\label{main}
For each $k\geq 1$, there exists a continuous linear right inverse for the deRham differential
$d \colon \Omega^{k-1}(M) \to d(\Omega^{k-1}(M))$.
\end{thm}
The idea of the proof is suggested in \cite[p.96]{MS} (although the reference there to 
an inductive procedure seems misleading).  A similar formula for the anti-differential
 appears in the PhD Thesis of Ioan Marcut 
\cite[Section 3.4.4]{Ma} under the added assumption that the manifold $M$ is of finite type
 (the details are spelled out for $k=2$). As pointed out in \cite{Ma}, the existence of a smooth
 family of primitives for a  smooth family of exact forms on an arbitrary manifold is stated as
 \cite[Lemma, p. 617]{GLSW} and a sketch proof is also given there.

One difference between the argument we give and the one described in \cite{Ma, GLSW} 
is that we make use of the \v{C}ech-deRham complex for 
cohomology with compact supports instead of the usual \v{C}ech-deRham complex. 
The proof will consist of using explicit quasi-isomorphisms (cf. \cite[I.9]{BT}) between the deRham complex on $M$, 
the \v{C}ech-deRham double complex with compact supports and a complex for \v{C}ech homology
with real coefficients to translate the problem of finding an anti-differential to the combinatorial problem of finding a
bounding chain in the \v{C}ech complex.

\subsection{The \v{C}ech-deRham double complex with compact supports}

Suppose the dimension of the manifold $M$ is $n$ and 
let $\mathcal \U=\{U_\alpha\}_{\alpha \in I}$ be a cover of $M$ such that 
the sets $\overline{U_\alpha}$ are compact,
every non-empty intersection of open subsets in $\mathcal U$ is diffeomorphic to $\R^n$, and 
every intersection of $n + 2$ elements of $\U$ is  empty (one can, for instance, cover $M$ by the open stars of 
vertices in a smooth triangulation). We pick a total ordering of the indexing set $I$ and consider the 
\v{C}ech-deRham double complex\footnote{Our notation below is meant to suggest \v{C}ech homology 
with coefficients in a cosheaf (cf. \cite[VI.4]{Br})
even though the \v{C}ech complexes we use are a completed version where direct sums are replaced with 
products.} for compactly supported cohomology given by
\begin{equation}
\label{cech}
C^{-p,q} = {\mathaccent 20 C}_{p}(\U; \Omega_c^q) = \prod_{\alpha_0<\cdots<\alpha_p} \Omega^q_c(U_{\alpha_0\cdots\alpha_p})
\end{equation}
with $p, q \in \{0,\ldots, n\}$. We'll write $\omega=(\omega_{\alpha_0 \cdots \alpha_p})$ for an element 
of $C^{-p,q}$ and use the convention  $\omega_{\beta_0 \cdots \beta_p} = 0$ if $\beta_i = \beta_j$ for some 
$i\neq j$. Moreover, if $\sigma$ is a permutation of  $\{\alpha_0,\ldots,\alpha_p\}$ with $
\alpha_0<\ldots<\alpha_p$, we set $\lambda_{\sigma(\alpha_0)\cdots \sigma(\alpha_p)} = (-1)^{\sgn(\sigma)} \lambda_
{\alpha_0 \cdots \alpha_p}$.

The horizontal and vertical differentials $d^h \colon C^{-p,q} \to C^{-(p-1),q}$ and $d^v \colon C^{-p,q} \to
C^{-p,q+1}$ are given by the formulas
\begin{equation}
\label{horvert}
(d^h \omega)_{\alpha_0\cdots \alpha_p} = \sum_{\alpha \in I} \omega_{\alpha\alpha_0 \cdots \alpha_p},
\quad \quad
(d^v \omega)_{\alpha_0 \cdots \alpha_p} = (-1)^p d\omega_{\alpha_0\cdots\alpha_p}.
\end{equation}

We can augment the complex $C^{*,*}$ in the horizontal direction by setting $ C^{1,*} = \Omega^*(M)$
and defining  $d^v= -d$ on $C^{1,*}$, and $d^h \colon C^{0,q} \to C^{1,q}$  by the formula
$$ d^h(\omega_\alpha) = \sum_{\alpha \in I} \omega_\alpha.$$
We can also augment the complex $C^{*,*}$ in the vertical direction by 
setting
\begin{equation}
\label{vertaug}
C^{-p,n+1}={\mathaccent 20 C}_{p}(\U; \R) = \prod_{\alpha_0<\cdots <\alpha_p} \R 
\end{equation}
with horizontal differential $d^h$ still given by \eqref{horvert}, and $d^v \colon C^{-p,n} \to C^{-p,n+1}$
given by  
$$ d^v(\omega_{\alpha_0\cdots \alpha_p}) = (-1)^p
 \left(\int_{U_{\alpha_0\cdots \alpha_p}} \omega_{\alpha_0\cdots \alpha_p}\right).
$$

\begin{figure}[h]
\begin{center}
\setlength{\unitlength}{.7mm}
\begin{picture}(200,60)(0,0)

\put(170,58){\makebox{$q$}}
\put(188, 5){\makebox{$p$}}
\put(172,5){\makebox{$\scriptstyle{1}$}}
\put(164,5){\makebox{$\scriptstyle{0}$}}
\put(32,5){\makebox{$\scriptstyle{-n}$}}
\put(12,5){\makebox{$\scriptstyle{-(n+1)}$}}
\put(172,13){\makebox{$\scriptstyle{0}$}}
\put(172,50){\makebox{$\scriptstyle{n+1}$}}
\put(172,40){\makebox{$\scriptstyle{n}$}}
\put(80,25){\makebox{${\mathaccent 20 C}_{-p}( \U ;\Omega^q_c)$}}
\put(80,50){\makebox{${\mathaccent 20 C}_{-p}(\U ; \R)$}}
\put(178,25){\makebox{$\Omega^q$}}
\put(30,0){\line(0,1){55}}
\put(10,45){\line(1,0){180}}
\put(10,10){\vector(1,0){180}}
\put(170,0){\vector(0,1){55}}

\end{picture}

\caption{The augmented \v{C}ech-deRham double complex.}
\end{center}
\end{figure}
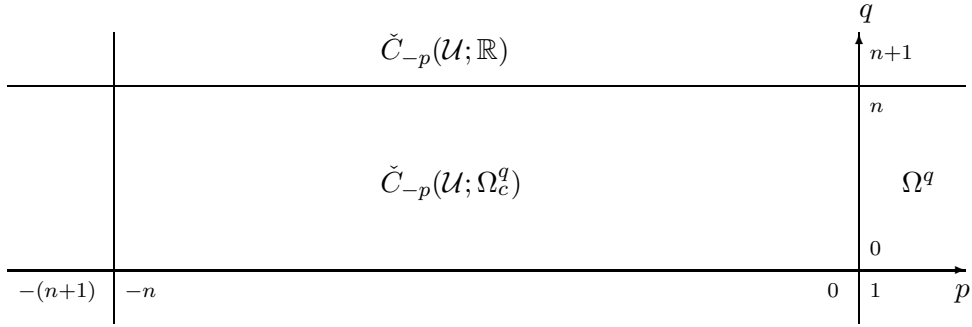

\subsection{Explicit contractions of rows and columns}

Let $\{\rho_\alpha\}_{\alpha \in I}$ be a partition of unity subordinate to the cover $\mathcal U$ and define 
$K \colon C^{-p,q} \to C^{-(p+1),q}$, for $0\leq q \leq n$ and $-n \leq  -p \leq 1$ by
\begin{equation}
\label{controws}
(K \omega)_{\alpha_0 \cdots \alpha_{p+1}} = \sum_{i=0}^{p+1} (-1)^i 
\rho_{\alpha_i} \omega_{\alpha_0 \cdots \hat{\alpha_i} \cdots \alpha_{p+1}}.
\end{equation}
One easily checks (cf. \cite[(8.7) p.95]{BT} for the case of the \v{C}ech-deRham complex with
 not necessarily compact supports) that  for each $q\in \{0,\ldots,n\}$, the operator 
 \eqref{controws} is a cochain contraction (i.e. $d^hK+Kd^h=\id$) 
 of the complexes $(C^{*,q},d^h)$ with $-n\leq * \leq 1$, 
 for each $q$ such that $0\leq q \leq n$.

Let $\pi \colon \R^n \to \R^{n-1}$ be the projection onto the first $(n-1)$-coordinates and $e(t)dt \in \Omega^1_c(\R)$
be a $1$-form with integral $1$. There are operators 
$$ \pi_* \colon \Omega^k_c(\R^n) \to \Omega^{k-1}_c(\R^{n-1}), \quad \text{ and } \quad
  e_* \colon \Omega^{k-1}_c(\R^{n-1}) \to \Omega^k_c(\R^n) $$
given respectively by integration along the fiber of $\pi$ and exterior product with $e(x_n)dx_n$. Multiplying 
by $(-1)^{*-1}$ the operator $K$ in \cite[Proposition I.4.6, p. 38]{BT}, we obtain operators 
$Q \colon \Omega^*_c(\R^n) \to \Omega^{*-1}_c(\R^n)$ such that $dQ + Qd = 1-e_*\pi_*$.
The cochain complex $\Omega^*_c(\R^n)$ can be augmented by $\R$ in degree $n+1$ by setting 
$d \colon \Omega^n_c(\R^n) \to \R$ to be $d \omega = \int_{\R^n} \omega = (\pi_*)^n \omega$, and it is 
then easy to check that $L \colon \Omega^*_c(\R^n) \to \Omega^{*-1}_c(\R^n)$ defined by 
\begin{equation}
\label{contcols}
 L = \begin{cases} \sum_{j=0}^{n-1} (e_*)^jQ(\pi_*)^j & \text{ if } * \leq n, \\
(e_*)^n & \text{ if } *=n+1,
\end{cases}
\end{equation}
(where we have identified $\R=\Omega^{n+1}_c(\R^n)$ with $\Omega^0_c(\R^0)$) 
is a cochain contraction of the augmented complex $\Omega^*_c(\R^n)$. Taking the product of these
cochain contractions we obtain cochain contractions (still denoted by $L$) of the complexes 
$(C^{-p,*},d^v)$ with $0\leq * \leq n+1$, for each $p$ such that $0 \leq p \leq n$. 

Let $C^* = \Tot (C^{*,*})$ denote the total cochain complex associated to the double complex $C^{p,q}$\
with $-n\leq p \leq 0, 0\leq q \leq n$.
Thus $C^k = \oplus_{i \leq 0} C^{i,k-i}$, and the differential in $C^*$ is given by 
$D=d^h+d^v$.  We write 
$$\omega^{(m)} \in C^{-m,k+m} \text{ for the components of } \omega \in C^k.$$

The cochain contractions of the rows and columns of the augmented double complex 
imply that we have quasi-isomorphisms
$$
{\mathaccent 20 C}_{-*}(\U;\R) \xleftarrow{I} C^* \xrightarrow{S} \Omega^*(M)
$$
given in degree $k$ by
$$ I(\omega) = \left( \int_{\R^n} \omega^{(n-k)}_{\alpha_0 \cdots \alpha_{n-k}}\right) \quad 
  \text{and} \quad S( \omega) = \sum_{\alpha \in I} \omega^{(0)}_\alpha. $$

We will use the following lemma to lift a cochain $x$ along $S$ or $I$ given we already have a lift of $dx$. 
In the statement below, the $N$-th column (with the negative of 
the vertical differential) is to be regarded as an augmentation of the double complex to its left.
 We leave the proof to the reader.
\begin{lemma}
\label{alglemma}
Let $(C^{p,q})_{p,q\in\Z}$ be a double cochain complex bounded on the right (i.e. $C^{p,q}=0$ for $p>N$)
with horizontal and vertical differentials $d^h$, $d^v$.
Assume $K\colon C^{p,*} \to C^{p-1,*}$  satisfy $d^hK+Kd^h = \id$.

Let $x \in C^{N,q}$ and suppose that $\alpha_i \in C^{N-i-1,q+i+1}$ for $i\geq 0$ are such that
\begin{enumerate} 
\item $d^h\alpha_{i+1} + d^v\alpha_i = 0$ for $i\geq 0$,
\item $d^h(\alpha_0) = - d^vx$.
\end{enumerate}
Then defining
$$ \beta_i = \sum_{j=1}^i (-1)^{j-1}(Kd^v)^{j-1}K\alpha_{i-j} + (-1)^i(Kd^v)^i Kx \in C^{N-i-1,q+i}  \text{ for } i\geq 0$$
we have
\begin{enumerate}
\item $d^h \beta_{i+1} + d^v\beta_{i} = \alpha_i$ for $i\geq 0$,
\item $d^h(\beta_0)=x$.
\end{enumerate}
\end{lemma}
\subsection{Proof of the theorem}
We note that the cochain contractions $K$, $L$ defined in \eqref{controws} and \eqref{contcols} are continuous 
when we give $\Omega^k(M)$ the weak Whitney topology, ${\mathaccent 20 C}_p(\U;\R)$ the product
topology and ${\mathaccent 20 C}_p(\U;\Omega^q_c)$ the product topology of the weak (or strong, in fact) 
Whitney topologies. Moreover the differentials $d^h$ and $d^v$ in the augmented double complex are
also continuous with respect to these topologies.

\begin{proof}[Proof of Theorem \ref{main}]
Let $\omega \in \Omega^k(M)$ be an exact form. By Lemma \ref{alglemma} (where we take $x=\omega$ and $\alpha_i=0$), 
\begin{equation}
\label{formulagamma}
\gamma = \sum_{i=0}^{n-k} \gamma^{(i)} =  \sum_{i=0}^{n-k}  (-1)^i (Kd^v)^iK\omega \in C^k
\end{equation}
is a cocycle in $C^*$, such that $S(\gamma) = \omega$. Since $S$ and $I$ are quasi-isomorphisms, we have 
that $I(\gamma) \in  {\mathaccent 20 C}_{n-k}(\U;\R)$ is a coboundary.

We pick linear maps 
$$
T \colon {\mathaccent 20 C}_{p-1}(\U;R) \to {\mathaccent 20 C}_{p}(U;\R)
$$
such that $d^h T d^h = d^h$ and each component of 
$T(c) \in {\mathaccent 20 C}_p(\U;\R) = \prod_{\beta_0 < \cdots < \beta_p} \R$
 depends only on \emph{finitely} many components of $c \in \prod_{\alpha_0 < \cdots < \alpha_{p-1}} \R$ 
 (so $T$ is continuous for the product topologies).
This is possible because the complex ${\mathaccent 20 C}_*(\U;\R)$ is the $\R$-dual of a chain 
complex
$$ \cdots \leftarrow \oplus_{\alpha_0 < \cdots < \alpha_{p-1}} \R \xrightarrow{\partial} \oplus_{\alpha_0 
< \cdots < \alpha_p}  \R \leftarrow \cdots$$
It suffices to pick any map $t$ such that $\partial t \partial = \partial$ and take $T = t^{\vee}$.

We can now apply Lemma \ref{alglemma} (with the horizontal and vertical axis interchanged) to 
$x=TI(\gamma)$ and $\alpha_i = (-1)^{n-k-1}\gamma^{(n-k-i)}$. It tells us that 
$$ \delta = \sum_{i=0}^{n-k+1} \delta^{(i)} = \sum_{i=0}^{n-k+1}  \left( \sum_{j=1}^{n-k+1-i} (-1)^{j-1} (Ld^h)^{j-1}L
\gamma^{(i+j-1)} + (-1)^{i}(Ld^h)^{n-k+1-i} LTI\gamma \right) $$
is a cochain in $C^{k-1}$ with $D \delta = \gamma$. The sought after primitive of $\omega$ is therefore
$S(\delta)$ which is explicitly given by the following formula in terms of the differentials in the extended double 
complex and the contraction operators for rows and columns
\begin{equation}
\label{formula}
d^h \left( \sum_{j=1}^{n-k+1} (-1)^{j-1} (Ld^h)^{j-1}L \gamma^{(j-1)} +(-1)^{n-k} (Ld^h)^{n-k+1} LTd^v\gamma^{(n-k)}
\right)
\end{equation}
where the $\gamma^{(i)}$ are given by \eqref{formulagamma}. This completes the proof.
\end{proof}

\begin{remark}
\begin{enumerate}[(i)]
\item It is clear from the formula \eqref{formula} for the anti-derivative that it
takes smooth families of exact forms to smooth families of forms. 
\item Suppose $M$ is a smooth manifold with boundary. Using a collaring, it is easy to obtain a
continuous anti-differential on $\Omega^*(M)$ from the one constructed in the proof of 
Theorem \ref{main} on $\Omega^*(\inter M)$.

\item  The proof of Theorem \ref{main} also gives a continuous linear inverse for  
$d \colon \Omega^{k-1}_c(M) \to d(\Omega^{k-1}_c(M))$. 
It suffices to replace the direct products in \eqref{cech} and \eqref{vertaug} by direct
sums. The existence of the splittings $T$ is immediate in this case.

\item A continuous right inverse for the deRham differential
does not in general exist if we consider the strong Whitney topology on spaces of forms
(it suffices to consider the case when $M=\R$).
\end{enumerate}
\end{remark}


\end{document}